\documentclass[12pt,a4paper]{article}
\usepackage{latexsym,amsfonts,amsmath,theorem,amssymb}
\usepackage{graphicx}
\usepackage{color}
\usepackage{enumerate}
\usepackage{hyperref}

\newtheorem{theorem}{Theorem}

\newtheorem{corollary}[theorem]{Corollary}
\newtheorem{proposition}[theorem]{Proposition}
\newtheorem{definition}{Definition}

\newenvironment{proof}{\begin{trivlist}
    \item[\hskip\labelsep{\bf Proof.}]}{$\hfill\Box$\end{trivlist}}

{\theoremstyle{plain} \theorembodyfont{\rmfamily}
\newtheorem{remark}{Remark}}
{\theoremstyle{plain} \theorembodyfont{\rmfamily}
\newtheorem{example}{Example}}

\allowdisplaybreaks

\newcommand{\bsgamma}{{\boldsymbol{\gamma}}}

\newcommand{\bst}{{\boldsymbol{t}}}

\newcommand{\bsx}{{\boldsymbol{x}}}

\newcommand{\bszero}{{\boldsymbol{0}}}

\newcommand{\rd}{\,\mathrm{d}}

\newcommand{\re}{\mathrm{e}}
\newcommand{\bbR}{\mathbb{R}}

\newcommand{\bbN}{\mathbb{N}}

\newcommand{\calA}{\mathcal{A}}

\newcommand{\calL}{\mathcal{L}}

\newcommand{\calF}{\mathcal{F}}
\newcommand{\calG}{\mathcal{G}}

\newcommand{\calI}{\mathcal{I}}

\newcommand{\calS}{\mathcal{S}}

\newcommand{\mask}[1]{}
\newcommand{\esup}{\operatornamewithlimits{ess\,sup}}

\newcommand{\e}{{\varepsilon}}
\newcommand{\setu}{{\mathfrak{u}}}
\newcommand{\setv}{{\mathfrak{v}}}
\newcommand{\setw}{{\mathfrak{w}}}

\newcommand{\norm}[1]{\left\Vert#1\right\Vert}
\newcommand{\abs}[1]{\left\vert#1\right\vert}
\newcommand{\sset}{[s]}

\newcommand{\dimtr}{{\rm dim}^{\rm trnc}}

\newcommand{\EXCLUDE}[1]{}

\date{}

\begin{document}

\title{Truncation Dimension for Linear Problems on Multivariate Function Spaces}

\author{Aicke Hinrichs\thanks{A. Hinrichs, P. Kritzer, and F. Pillichshammer gratefully acknowledge the support of the Erwin Schr\"odinger International Institute for Mathematics and Physics (ESI) in Vienna under the thematic programme ``Tractability of High Dimensional Problems and Discrepancy''.}, Peter Kritzer\thanks{P. Kritzer is supported by the Austrian
Science Fund (FWF), Project F5506-N26.}, Friedrich Pillichshammer\thanks{F. Pillichshammer is supported by the 
Austrian Science Fund (FWF) Project F5509-N26. Both projects are parts
of the Special Research Program "Quasi-Monte Carlo Methods:
Theory and Applications".}, G.W. Wasilkowski}

\maketitle

\abstract{The paper considers linear problems on weighted
spaces of multivariate functions of many variables. The main questions
addressed are: When is it possible to approximate
the solution for the
original function of very many variables by the solution for the same function; however with
all but the first $k$ variables set to zero, so that the corresponding
error is small? What is the {\em truncation dimension}, i.e., the
smallest number $k=k(\e)$ such that the
corresponding error is bounded by a given error demand $\e$?
Surprisingly, $k(\e)$ could be very small
even for weights with a modest speed of convergence to zero.}

\section{Introduction}
This paper is a continuation of our study initiated in  \cite{KrPiWa,KrPiWa16}
on the truncation dimension for functions with a huge (or even
infinite) number of variables. In \cite{KrPiWa} the problem of numerical integration and in \cite{KrPiWa16} function approximation is studied, but in both cases only anchored Sobolev spaces are considered. Here our focus is on more general linear problems and function spaces.

We start by providing the definition of the truncation dimension. Let $\calF$ be a normed linear space of $s$-variate functions
defined on $D^s$. Here
$s$ is very large or even infinite. We assume that $D$ is a possibly unbounded interval of $\bbR$ such that $0 \in D$.
Let
\[\calS\,:\,\calF\,\to\,\calG 
\]
be a continuous linear operator acting from $\calF$ to another normed linear space $\calG$.
Consider the problem of approximating $\calS(f)$ for $f \in \calF$. 
Suppose that for any $f\in\calF$ and any $k\in\bbN:=\{1,2,3,\ldots\}$
the function
$f_k$ obtained from $f$ by setting all the variables $x_j$ with $j>k$
to be zero, 
\begin{equation}\label{starstar}
f_k(\bsx)\,:=\,f(x_1,\dots,x_k,0,0,\dots) \ \ \ \mbox{for $\bsx=(x_1,x_2,\ldots,x_s)$,}
\end{equation}
also belongs to $\calF$.
It is natural to ask whether
approximations of $S(f_k)$,
where $k\ll s$ (here $\ll$ indicates in an
informal way that $k$ is ``much'' smaller than $s$), 
are good enough to approximate $\calS(f)$. This leads
to the following notion of {\em truncation error} and {\em truncation dimension}.

For given $k\in\bbN$, by the {\em $k^{th}$ truncation error} we
mean
\[
{\rm err^{trnc}}(k)\,:=\,\sup_{f\in\calF}
\frac{\|\calS(f-f_k)\|_{\calG}}{\|f-f_k\|_{\calF}}.
\]
\begin{definition}\label{def:truncdim}\rm
For a given error demand $\e>0$, the
{\em $\e$-truncation dimension for approximating $\calS$}
(or {\em truncation dimension} for short) is defined as
\[
{\rm dim^{trnc}}(\e)\,:=\,\min\left\{k\,:\,{\rm err^{trnc}}(k)\,\le\,\e
\right\}.   
\]
\end{definition}

It is the main aim of this work to relate the truncation dimension to the error of approximation of $\calS(f)$. We prove such a relation in Theorem~\ref{thm:main}. 

We stress that the truncation dimension is a property of the problem.
In particular, it depends on the spaces $\calF,\calG$, the operator
$\calS$, and the error demand $\e$. This is in  contrast with the
truncation dimension concept in statistical literature (see e.g.
\cite{CaMoOw,LiOw,Ow,WaKTF}), which depends
on the particular function under consideration. Moreover,
it is defined via ANOVA decomposition which 
is hard to approximate directly using only
a finite number of function values.

In the rest of the paper we estimate the truncation dimension for
a special, yet important, class of $\bsgamma$-weighted spaces
$\calF$ with {\em anchored decomposition} and $\calS$ having a
tensor product form. 

Roughly speaking (see Section~\ref{sec2} for details), functions from $\calF$ have a unique
decomposition
\[
f(\bsx)\,=\,\sum_\setu f_\setu(\bsx),
\]
where the sum is with respect to finite subsets 
$\setu\subseteq\{1,\dots,s\}$ (or $\setu\subset\bbN$ if $s=\infty$),
and each function $f_\setu$ depends only on the variables listed in 
$\setu$. Moreover, each $f_\setu$ belongs to a normed space $F_\setu$,
and we define the norm in $\calF$ by 
\[
   \|f\|_\calF\,=\,
\left(\sum_{\setu}\gamma_\setu^{-p}\,
\|f_\setu\|_{F_\setu}^p\right)^{1/p}<\infty
\]
for some $1\le p\le \infty$ and positive numbers $\gamma_\setu$
called weights. Of course $\|f\|_\calF=\sup_\setu \gamma_\setu^{-1}\,
\|f_\setu\|_{F_\setu}$ when $p=\infty$.

For $s=\infty$, these spaces include 
spaces of special functions of the form
\[
f(\bsx)\,=\,g({\rm X}(t)),
\]
where ${\rm X}(t)$ is the value of the stochastic
process 
${\rm X}\,=\,\sum_{i=1}^\infty x_i\,\phi_i$ at time $t$.
Here the $x_i$'s are i.i.d.~random variables and the base functions
$\phi_i$ converge to zero sufficiently fast.
Clearly
\[f(\bsx)\,=\,g\left(\sum_{i=1}^\infty x_i\,\phi_i(t)\right)
\quad\mbox{whereas}\quad
f_k(\bsx)\,=\,g\left(\sum_{i=1}^kx_i\,\phi_i(t)\right). 
\]

Functions of the form
\[
f(\bsx)=g\left(\sum_{i=1}^\infty x_i\,\phi_i(t)\right)
\]
appear in a number of applications including partial differential with
random coefficients and stochastic differential equation. The spaces
considered in the paper are more complex and, hence, our positive results
are even more important.

Assume that the weights
$\gamma_\setu$ can be written as product weights of the form
\[\gamma_\setu\,=\,\prod_{j\in\setu}\phi_j^\beta(t)
\]
for some $\beta\le1$. 
We prove that then  (cf.~Theorem \ref{thm:simple}) 
the truncation dimension can be bounded from above by
the smallest integer $k$ such that 
\[
\left(\prod_{j=1}^\infty(1+(C_1\phi_j^\beta (t))^{p^*}) 
\left(1-\exp\left(-C_1^{p^*} \sum_{j=k+1}^\infty 
(\phi_j^\beta (t))^{p^*}\right)\right)\right)^{1/p^*}
\le \varepsilon. 
\]
Here and elsewhere, $p^*$ 
is the conjugate of $p$ (i.e., $\frac{1}{p}+\frac{1}{p^*}=1$), $C_1>0$ is a number such that $\|S_1\| \le C_1$ 
and $\|S_1\|$ is the norm of the operator $\calS$ restricted to the
space $F_{\{1\}}$ of functions depending only on one variable.
(One has to apply the usual adaptions if $p=1$, cf.~Theorem
\ref{thm:simple} again.) From this result it can be seen that faster
decay of the $\phi_j^\beta(t)$ leads to smaller truncation
dimension.

We illustrate this for different values of $\e$, $p$, and 
$\gamma_\setu$.  We use product weights of the form
$\gamma_\setu=\prod_{j\in \setu} j^{-\alpha}$ for 
$\alpha \in \{2,3,4,5\}$. For simplicity, we assume that $C_1=1$.
For $p=2$ we have:

\[
\begin{array}{c||c|c|c|c|c|c||l}
\varepsilon &  10^{-1}  &  10^{-2}  & 10^{-3}  & 10^{-4}  & 10^{-5}  &  10^{-6} & \\
\hline
{\rm dim^{trnc}}(\varepsilon)& 4  &    19  & 90 & 416 & 1933  &  8973
&\alpha=2\\
\hline
{\rm dim^{trnc}}(\varepsilon)& 2   &  5   &  13   & 33     &  84      &  210
&\alpha=3\\
\hline
{\rm dim^{trnc}}(\varepsilon)& 2   &  3   &  6   & 12     &   22      &  43
&\alpha=4\\
\hline
{\rm dim^{trnc}}(\varepsilon)& 1   &  2   &  4   & 7     &  11       &  18
&\alpha=5
\end{array}
\] 
In particular, for the error demand $\e=10^{-3}$ it is enough to work
with only $90$ variables when $\alpha=2$, only $13$ variables when $\alpha=3$,
and with $6$ or $4$ when $\alpha=4$ or $5$, respectively. 

For $p=1$ we have ${\rm dim^{trnc}}(\varepsilon)=\lceil \varepsilon^{-1/\alpha}-1\rceil$, which leads to even better results. The following table already appeared in \cite{KrPiWa16}:
\[
\begin{array}{c||c|c|c|c|c|c||l}
\varepsilon &  10^{-1}  &  10^{-2}  & 10^{-3}  & 10^{-4}  & 10^{-5}  &  10^{-6} & \\
\hline
{\rm dim^{trnc}}(\varepsilon)& 3  &    9  & 31 & 99 & 316  &  999
&\alpha=2\\
\hline
{\rm dim^{trnc}}(\varepsilon)& 2   &  4   &  9   & 21     &  46      &  99
&\alpha=3\\
\hline
{\rm dim^{trnc}}(\varepsilon)& 1   &  3   &  5   & 9     &   17      &  31
&\alpha=4\\
\hline
{\rm dim^{trnc}}(\varepsilon)& 1   &  2   &  3   & 6     &  9        &  15
&\alpha=5
\end{array}
\] 

The content of the paper is as follows. In Section 2, we provide basic 
definitions and the main result. In Section 3, we propose spaces that 
are generalizations of anchored Sobolev spaces with bounded mixed 
derivatives of order one, that have been considered extensively 
in the literature. We next apply the general results to these 
special spaces. In Section 4, we study some unanchored spaces and show when
they are equivalent to their anchored counterparts.
Note that the equivalence implies that algorithms with small errors for
anchored spaces also have small errors for the corresponding
unanchored spaces. The results in Section 4
are extensions of results in \cite{GnHeHiRiWa16,HeRiWa15,SH,KrPiWa16a}
since they pertain to general spaces of this paper and more general
decompositions than the ANOVA one.

\section{Weighted Anchored Spaces of Multivariate Functions}\label{sec2}
We begin by introducing the notation used throughout the paper.
For $s \in \mathbb{N}$ and 
\[
  \sset\,:=\,\{1,2,\dots,s\},
\]
we will use $\setu,\setv$ to denote subsets of $\sset$, i.e., 
$\setu,\setv\subseteq\,\sset$. 
If $s=\infty$, then $[s]=\bbN$
and $\setu,\setv$ denote finite subsets of $\bbN$.

We assume that the functions $f\in\calF$ have a unique 
decomposition of the form
\begin{equation}\label{dec}
f\,=\,\sum_{\setu \subseteq [s]} f_\setu,
\end{equation}
where each $f_\setu$ belongs to a normed linear space $F_\setu$
such that $F_\setu\cap F_\setv=\{0\}$ if $\setu\not=\setv$,
$f_\setu$ depends only on $\bsx_\setu=(x_j)_{j\in\setu}$, and 
\begin{equation}\label{star}
f_\setu(\bsx)\,=\,0\quad\mbox{if $x_j=0$ for some $j\in\setu$}.
\end{equation}
Here $F_\emptyset$ is the space of constant
functions with the  absolute value as its norm.
In the case $s=\infty$,
the convergence of the series \eqref{dec} is with respect to
the norm in $\calF$ defined below in \eqref{haha}.
In some cases the series \eqref{dec} need not converge
pointwise; then we treat it as a sequence $(f_\setu)_\setu$ and
$\calF$ is a sequence space. 
We refer to Remark \ref{rem1} for a further discussion of the case $s=\infty$. 
Clearly, the property \eqref{star} yields that for any $f\in\calF$
and $k\in[s]$, 
\begin{equation}\label{dec-tr}
f_k(\bsx)\,=\,\sum_{\setu\subseteq [k]}f_\setu(\bsx),
\end{equation}
where $f_k$ is defined in \eqref{starstar}.

We assume that for given positive weights
$\bsgamma=(\gamma_\setu)_{\setu \subseteq [s]}$, the norm in
$\calF$ is given by
\begin{equation}\label{haha}
\|f\|_\calF\,=\,\left(\sum_{\setu \subseteq [s]}\gamma_\setu^{-p}\,
\|f_\setu\|_{F_\setu}^p\right)^{1/p}.
\end{equation}

Let $S_\setu$ be $\calS$ restricted to $F_\setu$, and let
$\|S_\setu\|$ be its operator norm,
\[
\|S_\setu\|\,:=\,\sup_{f_\setu\in F_\setu}
\frac{\|S_\setu(f_\setu)\|_\calG}{\|f_\setu\|_{F_\setu}}.
\]
We have the following simple proposition.

\begin{proposition}\label{prop:simple} For every $k\le s$ we have
\[  {\rm err^{trnc}}(k)\,\le\,\left(\sum_{\setu\not\subseteq[k]}
\|S_\setu\|^{p^*}\,\gamma_\setu^{p^*}\right)^{1/p^*},
\]
where here and throughout this paper $\sum_{\setu\not\subseteq[k]}$ 
means summation over all $\setu \subseteq [s]$ with $\setu\not\subseteq[k]$. 
Hence
$${\rm dim}^{\rm trnc}(\varepsilon) \le \min\left\{k \ : \ \left(\sum_{\setu\not\subseteq[k]}
\|S_\setu\|^{p^*}\,\gamma_\setu^{p^*}\right)^{1/p^*}\le \varepsilon\right\}.
$$
Of course, for $p^*=\infty$, we have
\[  {\rm err}^{\rm trnc}(k)\le\sup_{\setu\not\subseteq[k]}\|S_\setu\|\,
\gamma_\setu
\quad\mbox{and}\quad
{\rm dim}^{\rm trnc}(\e)\le \min\left\{k\,:\,\sup_{\setu\not\subseteq[k]}
\|S_\setu\|\,\gamma_\setu \le \e\right\}.          
\]
\end{proposition}

\begin{proof}
We have 
\begin{align*}
\|\mathcal{S}(f)- \mathcal{S}(f_k)\|_{\mathcal{G}} = & \left\|\sum_{\setu \not\subseteq [k]} S_\setu(f_\setu)\right\|_{\mathcal{G}} \le 
\sum_{\setu \not\subseteq [k]} \gamma_\setu \|S_\setu\| \gamma_\setu^{-1} \|f_\setu\|_{F_\setu}\\
\le &\left(\sum_{\setu\not\subseteq[k]}
\|S_\setu\|^{p^*}\,\gamma_\setu^{p^*}\right)^{1/p^*} \|f-f_k\|_{\mathcal{F}},
\end{align*}
where, in the last step, we used 
$\|f-f_k\|_\calF=(\sum_{\setu\not\subseteq[k]}\gamma_\setu^{-p}\|f_\setu\|_{F_\setu}^p)^{1/p}$ together with H\"older's inequality.
From this the result follows.
\end{proof}

In this paper we mainly concentrate on {\em product weights},
introduced in \cite{SlWo}, that have the form
\[
  \gamma_{\setu}\,=\,\prod_{j\in\setu}\gamma_j
\]
for a nonincreasing sequence of positive numbers $\gamma_j$ for
$j \in \mathbb{N}$.

\begin{theorem}\label{thm:simple}
Suppose that the weights $\bsgamma$ are product weights and that there exists 
a constant $C_1$ such that 
\begin{equation}\label{ass:tnspr}
\|S_\setu\|\,\le\,C_1^{|\setu|} \quad\mbox{for all\ }\setu.
\end{equation}

For $p>1$ and every $k\le s$  we have 
\[
{\rm err^{trnc}}(k)\,\le\,
\left(\prod_{j=1}^s(1+(C_1\,\gamma_j)^{p^*})\,
\left(1-\re^{-C_1^{p^*} \sum_{j=k+1}^s \gamma_j^{p^*}}
\right)\right)^{1/p^*}.
\]
Hence ${\rm dim}^{\rm trnc}(\e)$ is bounded from above by 
\[\min
  \left\{k\ :\ \prod_{j=1}^s(1+(C_1\,\gamma_j)^{p^*})^{1/p^*}\,
 \left(1-\re^{-C_1^{p^*}\sum_{j=k+1}^s\gamma_j^{p^*}}
  \right)^{1/p^*}\,\le\,\e\right\}.
\]

For $p=1$ and every $k\le s$ we have 
\[{\rm err}^{\rm trnc}(k)\,\le\,\max_{\setu\not\subseteq[k]}C_1^{|\setu|}
\,\gamma_\setu
\]
and if additionally $C_1\le 1$ then 
${\rm err}^{\rm trnc}(k)\le C_1\,\gamma_{k+1}$. 
\end{theorem}
\begin{proof}
We have 
\begin{eqnarray*}
\sum_{\setu\not\subseteq[k]}\|S_\setu\|^{p^*}\,\gamma_\setu^{p^*}
&\le&\sum_{\setu\not\subseteq[k]} C_1^{|\setu|\,p^*}\,\gamma_\setu^{p^*}\\
&=&\sum_{\setu\subseteq[s]} \prod_{j \in \setu} (C_1\, \gamma_j)^{p^*} - 
\sum_{\setu\subseteq[k]} \prod_{j \in \setu} (C_1\, \gamma_j)^{p^*}\\
&=& \prod_{j=1}^s(1+(C_1\,\gamma_j)^{p^*})-\prod_{j=1}^k(1+(C_1\,\gamma_j)^{p^*})\\
&=& \prod_{j=1}^s(1+(C_1\,\gamma_j)^{p^*})
\left(1-\re^{-\sum_{j=k+1}^s \ln(1+(C_1\, \gamma_j)^{p^*})}\right)\\
& \le & \prod_{j=1}^s(1+(C_1\, \gamma_j)^{p^*})
\left(1-\re^{-C_1^{p^*} \sum_{j=k+1}^s  \gamma_j^{p^*}}\right),
\end{eqnarray*}
since $\ln(1+x) \le x$ for all $x >-1$.
From this the result follows.
\end{proof}

\begin{remark}
It is well known that \eqref{ass:tnspr} holds 
if $F_1$ is a Hilbert space and,
for every $\setu\not=\emptyset$, the 
spaces $F_\setu$ are $|\setu|$-fold tensor products of $F_1$ and
also $S_\setu$ are $|\setu|$-fold tensor products of $S_1$, i.e.,
\[S_\setu\left(\bigotimes_{j\in\setu} f_j\right)\,=\,\bigotimes_{j\in\setu}
S_1(f_j).
\]
Actually then we have 
\[
\|S_\setu\|\,=\,C_1^{|\setu|}\quad\mbox{and}\quad \|S_1\|=C_1.
\]
However, \eqref{ass:tnspr} also holds with inequality 
for Banach spaces $F_\setu$ and operators $S_\setu$ that we
consider in the next sections. 
\end{remark}

We introduce some further notation. For
$\bsx=(x_1,x_2,\ldots,x_s)\in D^s$ and 
$\setu\subseteq\sset$, $[\bsx_\setu;\bszero_{-\setu}]$ denotes the
$s$-dimensional vector with all $x_j$ for $j\notin\setu$ replaced
by zero, i.e.,
\[
  [\bsx_\setu;\bszero_{-\setu}]\,=\,(y_1,y_2,\dots,y_s)\quad
   \mbox{with}\quad  y_j\,=\,\left\{\begin{array}{ll} x_j &
    \mbox{if\ }j\in\setu,\\ 0 & \mbox{if\ } j\notin\setu.\end{array}
    \right.
\]

As shown in \cite{KrPiWa} for the integration problem, the importance 
of the $\e$-truncation
dimension lies in the fact that when approximating $\mathcal{S}(f)$ for
functions $f \in \mathcal{F}$ it is sufficient to approximate $\mathcal{S}$ only for $k$-variate functions   
\[
 f_k(\bsx)\,=\,f(x_1,\ldots,x_k,0,0,\dots)\,=\,f([\bsx_{[k]};\bszero_{-[k]}])
\]
with $k\ge \dimtr(\e)$ since
$f-f_k=\sum_{\setu\not\subseteq[k]}f_\setu$ and, therefore, 
\[
\left\|\mathcal{S}(f)-\mathcal{S}(f_k)\right\|_{\calG}\,\le\,\e\,
\left\|\sum_{\setu\not\subseteq[k]}
f_\setu\right\|_{\calF}.
\]

For $k \le s$, let 
\[
 \calF_{k}=\bigoplus_{\setu\subseteq[k]}F_\setu
\]
be the subspace of $\calF$ consisting of $k$-variate 
functions $f([\cdot_{[k]};\bszero_{-[k]}])$, and let 
$A_{k,n}$ be an algorithm for approximating $\calS_{[k]}(f)$
for functions from 
$\calF_{k}$ that uses $n$ function values. The worst case error
of $A_{k,n}$ with respect to the space $\calF_{k}$ is 
\[
  e(A_{k,n};\calF_k):=
\sup_{f\in \calF_k}\frac{\|\calS_{[k]}(f)-A_{k,n}(f)\|_{\calG}}
{\|f\|_{\calF_k}}.
\]
Now let 
\[
  \calA^{\rm trnc}_{s,k,n}(f)\,=\,A_{k,n}(f([\cdot_{[k]};\bszero_{-[k]}]))
\]
be an algorithm for  approximating functions from the whole space
$\calF$. The worst case error of $\calA^{\rm trnc}_{s,k,n}$ is defined as 
\[
e(\calA^{\rm trnc}_{s,k,n};\calF)\,:=\,
\sup_{f\in\calF}
\frac{\|\calS(f)-\calA^{\rm trnc}_{s,k,n}(f)\|_{\calG}}
{\|f\|_{\calF}}.
\]

This yields the following theorem.

\begin{theorem}\label{thm:main}
For given $\e>0$ and  $k\ge \dimtr(\e)$ we have 
\[
e(\calA^{\rm trnc}_{s,k,n};\calF)\,\le\,
\left(\e^{p^*}+e(A_{k,n};\calF_k)^{p^*}\right)^{1/p^*}.
\]
\end{theorem}

\begin{proof}
The spaces $\calF_k$ are subspaces of
$\calF$. Moreover any
$f_k=f([\cdot_{[k]};\bszero_{-[k]}])$
belongs to $\calF_k$ and 
\[
  \|f_k\|_{\calF_k}\,=\,\|f_k\|_{\calF}.
\]
Therefore, for any $f\in\calF$ we have 
\begin{eqnarray*}
\|\calS(f)-\calA^{\rm trnc}_{s,k,n}(f)\|_{\calG}
&\le&\|\calS(f_k)-
A_{k,n}(f_k)\|_{\calG}
+\|\calS(f)-\calS(f_k)\|_{\calG}\\
&\le& e(A_{k,n};\calF_k)\,
\left\|\sum_{\setu\subseteq[k]}f_\setu\right\|_{\calF}
+\e\,\left\|\sum_{\setu\not\subseteq[k]}f_\setu
\right\|_{\calF}\\  &\le&
\left(\e^{p^*}+e(A_{k,n};\calF_{k})^{p^*}\right)^{1/p^*}\,\|f\|_{\calF},
\end{eqnarray*}
with the last inequality due to H\"older's inequality. 
\end{proof}

In the following section we consider anchored Sobolev spaces of
multivariate functions and show that the assumptions above are justified. 
As examples for the linear approximation problem we consider function
approximation and integration.

\section{Anchored Spaces of Multivariate Functions}\label{defspace}
In this section, we begin by recalling the definitions and basic properties 
of weighted anchored  Sobolev spaces of
$s$-variate functions with mixed partial derivatives of order one
bounded in $L_p$-norm. More
detailed information can be found in \cite{HeRiWa15,SH,Was14}.
Such spaces have often been assumed in the
context of {\em quasi-Monte Carlo methods}. However, for us
they serve as a motivation to consider more general classes of
anchored spaces. 

\subsection{Anchored Sobolev Spaces}\label{sec1.1}
Here we follow \cite{HeRiWa15}. We use the notations $[s]$, and $\setu,\setv \subseteq [s]$ as above. 
We also write $\bsx_\setu$ to denote the $|\setu|$-dimensional
vector $(x_j)_{j\in\setu}$ and
\[
  f^{(\setu)}\,=\,\frac{\partial^{|\setu|} f}{\partial \bsx_{\setu}}\,
=\,\prod_{j\in\setu}\frac{\partial}{\partial x_j}\,f
  \quad\mbox{with}\quad f^{(\emptyset)}=f.
\]

For a family of weights $\bsgamma=(\gamma_\setu)_{\setu \subseteq [s]}$, which are non-negative
numbers, and for $p\in[1,\infty]$ the corresponding
$\bsgamma$-weighted {\em anchored} space $\calF_{s,p,\bsgamma}$
is the Banach space of functions defined on $D^s=[0,1]^s$ 
with the norm
\[
\|f\|_{\calF_{s,p,\bsgamma}}\,=\,
\left(\sum_{\setu \subseteq [s]}\gamma_\setu^{-p}\,\|f^{(\setu)}([\cdot;
  \bszero_{-\setu}])\|^p_{L_p([0,1]^{|\setu|})}\right)^{1/p}.
\]
For $p=\infty$, the norm reduces to
\[
\|f\|_{\calF_{s,p,\bsgamma}}\,=\,\sup_{\setu \subseteq [s]}
\gamma_\setu^{-1}\,\|f^{(\setu)}([\cdot_\setu;\bszero_{-\setu}])
\|_{L_\infty([0,1]^{|\setu|})}.
\]

As shown in \cite{HeRiWa15} the functions from $\calF_{s,p,\bsgamma}$
have the unique decomposition
\[
f\,=\,\sum_{\setu \subseteq [s]} f_{\setu},
\]
where each $f_\setu$, although formally a function of $\bsx$,
depends only on the variables $\bsx_\setu$, and is an element of a space
$F_\setu$ given by
\[
F_\setu\,=\,K_\setu(L_p([0,1]^{|\setu|})).
\]
Here, for $\setu\not=\emptyset$ and $g_{\setu} \in L_p([0,1]^{|\setu|})$ 
\[
f_\setu(\bsx)\,=\,K_\setu(g_\setu)(\bsx)\,:=\,
\int_{[0,1]^{|\setu|}}g_\setu(\bst_\setu)\,
\prod_{j\in\setu}(x_j-t_j)^0_+\rd\bst_\setu,
\]
where $(x-t)^0_+=1$ if $t<x$ and $(x-t)^0_+=0$ otherwise, and
\[
\|f_\setu\|_{F_\setu}\,=\,\|g_\setu\|_{L_p([0,1]^{|\setu|})}.
\]
Recall that for $\setu=\emptyset$, $F_\setu$ is the space of constant functions with
the absolute value as its norm.

An important property of these spaces is that they are anchored at $0$,
i.e., for any $\setu\not=\emptyset$ and any $f_\setu\in F_\setu$,
\[
  f_\setu(\bsx)\,=\,0\quad\mbox{if\ $x_j=0$ for some $j\in\setu$}.
\]
This implies that
\[
f^{(\setu)}([\bsx_\setu;\bszero_{-\setu}])\,=\,f_\setu^{(\setu)}
\quad\mbox{and}\quad\|f\|_{\calF_{s,p,\bsgamma}}\,=\,
\left(\sum_{\setu \subseteq [s]} \gamma_{\setu}^{-p}\|f_\setu\|_{F_\setu}^p\right)^{1/p}.
\]

\subsection{More General Anchored Spaces}\label{secgeneral}
In this section we extend the definition of $\calF_{s,p,\bsgamma}$
from the previous section to spaces of functions
\[
f\,=\,\sum_{\setu \subseteq [s]} f_\setu
\]
with the components $f_\setu$ given by
\[
f_\setu\,=\,K_\setu(g_\setu)\,=\,\int_{D^{|\setu|}}
g_\setu(\bst_\setu)\,\kappa_\setu(\cdot_\setu,\bst_\setu)\rd\bst_\setu
\quad\mbox{with}\quad \kappa_\setu(\bsx_\setu,\bst_\setu)
\,=\,\prod_{j\in\setu}\kappa(x_j,t_j),
\]
where $\kappa(x,t)$ could be more general than $(x-t)^0_+$,  
and $g_\setu$ could be from a more general $\psi$-weighted
$L_p$ space.

More specifically, let $D$ be an interval in $\bbR$ that, without
any loss of generality, contains~$0$. This includes both bounded
intervals like $D=[0,1]$ from the previous subsection, 
as well as unbounded ones, e.g., $D=[0,\infty)$ or $D=\bbR$.

Let
\[
\psi\,:\,D\,\to\,\bbR_+  
\]
be a measurable and (a.e.) positive {\em weight} function.
For $p_1\in[1,\infty]$, 
by $L_{p_1,\psi}=L_{p_1,\psi}(D)$ we denote the space of scalar 
functions with the norm
\[
\|g\|_{L_{p_1,\psi}}\,=\,\left(\int_D|g(t)|^{p_1}\,
\psi(t)\rd t\right)^{1/p_1}.
\]
For non-empty $\setu$, $L_{\setu,p_1,\psi}=L_{\setu,p_1,\psi}(D^{|\setu|})$ 
is the space of $|\setu|$-variate functions with the norm given by
\[
\|g_\setu\|_{L_{\setu,p_1,\psi}}\,=\,
\left(\int_{D^{|\setu|}}|g_\setu(\bst_\setu)|^{p_1}\,
\psi_\setu(\bst_\setu)\rd\bst_\setu\right)^{1/p_1}
\quad\mbox{with}\quad\psi_\setu(\bst_\setu)\,=\,
\prod_{j\in\setu}\psi(t_j).
\]

Let  
\[
\kappa\,:\,D\times D\,\to\,\bbR 
\]
be a given measurable function.
For non-empty $\setu$, define
\[
\kappa_\setu(\bsx_\setu,\bst_\setu)\,:=\,
\prod_{j\in\setu}\kappa(x_j,t_j)
\]
and
\[
K_\setu(g_\setu)(\bsx_\setu)\,:=\,\int_{D^{|\setu|}}g_\setu(\bst_\setu)\,
\kappa_\setu(\bsx_\setu,\bst_\setu)\rd\bst_\setu\quad\mbox{for}\quad
g_\setu\in L_{\setu,p_1,\psi}.
\]
We assume that
\begin{equation}\label{ass:def}
\widehat{\kappa}_{p_1}(x)\,:=\,
\left(\int_D\left(\frac{|\kappa(x,t)|}{\psi^{1/p_1}(t)}\right)^{p_1^*}\rd t
\right)^{1/p_1^*}\,<\,\infty\quad\mbox{for all\ }x\in D.
\end{equation}
Of course, for $p_1=1$, 
$\widehat{\kappa}_1(x)=\esup_{t\in D}|\kappa(x,t)|/\psi(t)<\infty$ 
for all $x\in D$. 
Then
\[
  f_\setu(\bsx)\,=\,K_\setu(g_\setu)(\bsx)\quad(g\in L_{\setu,p_1,\psi})
\]
are well defined functions since
\[
|f_\setu(\bsx)|\,\le\,\|g_\setu\|_{L_{\setu,p_1,\psi}}\,
\widehat{\kappa}_{\setu,p_1}(\bsx),\quad\mbox{where}\quad
\widehat{\kappa}_{\setu,p_1}(\bsx)\,:=\,
\prod_{j\in\setu}\widehat{\kappa}_{p_1}(x_j).
\]
We also assume that $K_\setu$ is an injective operator, i.e.,
\begin{equation}\label{ass:1-1}
  K_\setu(g_\setu)\,\equiv\,0\quad\mbox{implies}\quad g_\setu\,=\,0
  \mbox{\ \ a.e.}
\end{equation}

We define the following Banach spaces
\[
F_\setu\,=\,K_{\setu}(L_{\setu,p_1,\psi}(D^{|\setu|}))\quad\mbox{with}\quad
\|K_\setu(g_\setu)\|_{F_\setu}\,:=
\,\|g_\setu\|_{L_{\setu,p_1,\psi}}.
\]
We assume also that 
\begin{equation}\label{ass:anchor}
\kappa(0,\cdot)\,\equiv\,0.
\end{equation}
Then the spaces $F_\setu$ are anchored at zero since for every
$f_\setu\in F_\setu$ we have
\[
  f_\setu(\bsx)\,=\,0\quad\mbox{if $x_j=0$ for some $j\in\setu$}.
\]
As in the previous section, $F_\emptyset$ is the space of constant
functions.

Finally, for $p_2\,\in\,[1,\infty]$, 
consider the Banach space $\calF_{s,p_1,p_2,\bsgamma}$
\[
\calF_{s,p_1,p_2,\bsgamma}\,=\,\bigoplus_{\setu \subseteq [s]} F_\setu\quad
\mbox{of functions}
\quad f\,=\,\sum_{\setu \subseteq [s]} f_{\setu} ,\mbox{\ where\ }f_\setu\in
F_\setu,
\]
with the norm given by 
\begin{equation}\label{def:norm}
\quad\|f\|_{\calF_{s,p_1,p_2,\bsgamma}}\,:=\,
\left(\sum_{\setu \subseteq [s]} \gamma_\setu^{-p_2}\,\|f_\setu\|_{F_\setu}^{p_2}
\right)^{1/p_2}.
\end{equation}

\begin{remark}\label{rem1} For functions with infinitely many variables,
$\calF_{\infty,p_1,p_2,\bsgamma}$ is the completion 
of  $\bigcup_\setu F_\setu$ with respect to the norm
\eqref{def:norm}. In general, it is a space of sequences
$f=\left(f_\setu\right)_{\setu}$ since $\sum_{\setu}f_\setu(\bsx)$
may not exist when $\bsx$ has infinitely many non-zero $x_j$'s.
Of course, it exists for $\bsx=[\bsx_\setu;\bszero_{-\setu}]$.
However, $\calF_{\infty,p_1,p_2,\bsgamma}$ is a function space if 
\begin{equation}\label{ass:inft}
\left(\sum_\setu\left(\gamma_\setu\,
\widehat{\kappa}_{\setu,p_1}(\bsx)
\right)^{p_2^*}\right)^{1/p_2^*}\,<\,\infty\quad
\mbox{for all\ }\bsx\in D^{\bbN}
\end{equation}
since then
\[
\sum_{\setu} |f_\setu(\bsx)|\,\le\,
\|f\|_{\calF_{\infty,p_1,p_2,\bsgamma}}\,\,
\left(\sum_\setu\left(\gamma_\setu\,
\widehat{\kappa}_{\setu,p_1}(\bsx)
\right)^{p_2^*}\right)^{1/p_2^*}\,<\,\infty
\]
shows that $f(x) = \sum_{\setu} f_\setu(\bsx)$ is well defined for every $\bsx\in D^{\bbN}$. 
\end{remark}

Here and elsewhere we use $p^*_i$ to denote the conjugate of $p_i$,
i.e., $1/p^*_i+1/p_i=1$.
  
We end this section with the following examples.

\begin{example}\label{exmp1}

  As in Section~\ref{sec1.1}, $D=[0,1]$, $\psi\equiv1$ and
$\kappa(x,t)=(x-t)^0_+$. Then the assumptions 
\eqref{ass:def}--\eqref{ass:anchor} are satisfied and
\[
  \widehat{\kappa}_{p_1}(x)\,=\,x^{1/p_1^*}.
\]
Moreover, for product weights and $s=\infty$, \eqref{ass:inft} holds iff
$\sum_{j=1}^\infty \gamma_j^{p^*_2}<\infty$ (or 
$\sup_\setu\gamma_\setu<\infty$ if $p_2^*=\infty$) 
since
\[
\left(\sum_\setu(\gamma_\setu\,\widehat{\kappa}_{\setu,p_1}(\bsx))^{p^*_2}
\right)^{1/p^*_2}\,=\,
\prod_{j=1}^\infty\left(1+\gamma_j^{p^*_2}\,x_j^{p^*_2/p^*_1}
\right)^{1/p^*_2}.
\]
\end{example}

\begin{example}\label{exmp3}
Let $D=[0,\infty)$, $\kappa(x,t)=(x-t)^{r-1}_+/(r-1)!$ for $r\ge 1$,
and $\psi(t)=\re^{\lambda t}$ for given $\lambda\in\bbR$.
Recall that $(x-t)^{r-1}_+=(\max(0,x-t))^{r-1}$.
For $s=\infty$, the space 
$\calF_{\infty,p_1,p_2,\bsgamma}$ is a sequence space. Hence we
consider here only finite $s$. 

For $p_1=1$, $\widehat{\kappa}_1(x)\,(r-1)!=\max_{t\in[0,x]}
(x-t)^{r-1}\,\re^{-\lambda t}$. For $\lambda\ge0$ or $x\le(r-1)/|\lambda|$,
the maximum above is attained at $t=0$. Otherwise, it is attained at
$t=x-(r-1)/|\lambda|$. Hence
\[\widehat{\kappa}_{1}(x)\,=\,\left\{\begin{array}{ll} \displaystyle
\frac{x^{r-1}}{(r-1)!} &\mbox{if $\lambda\ge0$ or
  $x\le\frac{r-1}{|\lambda|}$},\\
\displaystyle
\frac{(r-1)^{r-1}\,\re^{|\lambda|x-(r-1)}}{|\lambda|^{r-1}\,(r-1)!}
&\mbox{if $\lambda<0$ and $x>\frac{r-1}{|\lambda|}$}.
\end{array}\right.
\]
For $p_1>1$, $p^*_1<\infty$ and $p_1^*/p_1=p_1^*-1$. Hence 
\[\widehat{\kappa}_{p_1}(x)\,=\,
\frac1{(r-1)!}\left(\int_0^x (x-t)^{(r-1)p_1^*}\,
\re^{-\lambda\,t\,(p_1^*-1)}\rd t\right)^{1/p_1^*}.
\]
If $\lambda\ge0$, then
\[\widehat{\kappa}_{p_1}(x)\,\le\,
\frac{x^{r-1/p_1}}{(r-1)!\,((r-1)p_1^*+1)^{1/p_1^*}}.
\]
If $\lambda<0$, then
\[\widehat{\kappa}_{p_1}(x)\,(r-1)!\,\le\,\left(x^{(r-1)\,p_1^*}\,
\frac{\re^{|\lambda|\,x\,(p_1^*-1)}}{|\lambda|\,(p_1^*-1)}\right)^{1/p_1^*}
\]
and
\[\widehat{\kappa}_{p_1}(x)\,(r-1)!\,\le\,\left(\re^{|\lambda|\,x\,(p_1^*-1)}\,
\frac{x^{(r-1)p_1^*+1}}{(r-1)\,p_1^*+1}\right)^{1/p_1^*}.
\]
Hence, for $\lambda<0$,
\[\widehat{\kappa}_{p_1}(x)\,\le\,\re^{|\lambda|\,x/p_1}\,
\frac{x^{r-1}}{(r-1)!}\,
\left(\min\left(\frac{x}{(r-1)\,p_1^*+1}\,,\,\frac1{|\lambda|\,(p_1^*-1)}
\right)\right)^{1/p_1^*}.
\]

For this example, $f(0)=f'(0)=\dots=f^{(r-1)}(0)=0$. Our result also
holds for functions of the form
$f(x)\,=\,\sum_{j=1}^{r-1} a_j\,x^j/j!+K_1(g)$ 
with the norm changed to
$\|f\|\,=\,(\sum_{j=1}^{r-1}|f^{(j)}(0)|^{p_2}+
\|g\|_{L_{p_1,\psi}}^{p_2})^{1/{p_2}}$.

For $r=1$ and $p_1>1$, one can get exact values 
\[\widehat{\kappa}_{p_1}(x)\,=\,\left\{
\begin{array}{ll} x^{1/p^*} &\mbox{if\ }\lambda=0,\\ 
  \displaystyle \left(\frac{p_1-1}\lambda\right)^{1/p_1^*}\,
  \left(1-\re^{-\lambda\,x/(p_1-1)}\right)^{1/p_1^*} &\mbox{if\ }\lambda>0,
  \\ 
  \displaystyle \left(\frac{p_1-1}\lambda\right)^{1/p_1^*}\,
  \left(\re^{|\lambda|\,x/(p_1-1)}-1\right)^{1/p_1^*} &\mbox{if\ }\lambda<0.
\end{array}\right.
\]
\end{example}

\begin{example}\label{exam5}
Consider $D=[0,\infty)$ and
\[\kappa(x,t)\,=\,G(xt)
\]
for a smooth function $G$ with $G(0)=0$ and
$\|G^{(n)}\|_{L_\infty}< \infty$ for all $n$.
Then the functions 
\[f(x)\,=\,\int_0^\infty \kappa(x,t)\,g(t)\rd t\,=\,\int_0^\infty
G(xt)\,g(t)\rd t
\]
with $g \in L_{p_1,\psi}$ have all derivatives continuous given by
\[
  f^{(n)}(x)\,=\,\int_0^\infty G^{(n)}(x\,t)\,t^n\,g(t)\rd t
\]
provided that 
\begin{equation}\label{condxy}
\left(\int_0^\infty\left|\frac{t^n}{\psi^{1/p_1}(t)}
\right|^{p_1^*}\rd t\right)^{1/p_1^*}\,<\,\infty
\quad\mbox{for all $n$},
\end{equation}
and $\|G^{(n)}\|_{L_\infty}< \infty$, e.g., for $G(y)=1-\re^{-y}$ or $G(y)=1-\cos(y)$.

Indeed, consider first $n=1$. Then
\begin{eqnarray*}
  f'(x)&=&\lim_{n\to\infty}\int_0^\infty g(t)\,\frac{G((x+1/n)\,t)-G(x\,t)}{1/n}\,\rd t\\
  &=& \lim_{n\to\infty}\int_{0^+}^\infty g(t)\,t\,
  \frac{G(x\,t+t/n)-G(x\,t)}{t/n}\,\rd t\\
  &=&\int_0^\infty g(t)\,t\,G'(x\, t)\rd t.
\end{eqnarray*}
The last equality holds due to the dominated convergence theorem because 
\[
\lim_{n\to\infty}g(t)\,t\,\frac{G(x\,t+t/n)-G(x\,t)}{t/n}\,=\,
g(t)\,t\,G'(x\,t),
\]
\[
\left|g(t)\,t\,\frac{G(x\,t + t/n)-G(x\,t)}{t/n}
\right|\,\le\,|g(t)\,t|\,\|G'\|_{L_\infty}\quad\mbox{due to mean value theorem},
\]
and $|g(t)\,t|$ is integrable, since by H\"older's inequality and \eqref{condxy}
\begin{eqnarray*}
\int_0^{\infty} | g(t) t| \rd t \le \|g\|_{L_{p_1,\psi}}  \left(\int_0^\infty\left|\frac{t}{\psi^{1/p_1}(t)}
\right|^{p_1^*}\rd t\right)^{1/p_1^*} < \infty.
\end{eqnarray*}

The proof for an arbitrary $n$ is by induction. Since the inductive step is very similar
to te basic one for $n=1$, we omit it. 

Assume  additionally that $G^{(n)}(0)\not=0$ for all $n=1,2,\dots$
and that there exists $x$ such that
\begin{equation}\label{ass:hateit}
\int_0^\infty G(x\,t)\,\rd t\,\not=\,0.
\end{equation}
Then \eqref{ass:1-1} holds. Indeed, consider $g$ such that
\[
f\,=\,\int_0^\infty G(\cdot\,t)\,g(t)\,\rd t\,\equiv\,0. 
\]
Then 
\[0 = f^{(n)}(0)\,=\,G^{(n)}(0)\,\int_0^\infty t^n\,g(t)\,\rd t.
\]
i.e., $g$ is orthogonal to all polynomials $p(t)=t^n$ for $n=1,2,\dots$, i.e., $g$ is constant. However then, due to \eqref{ass:hateit},
$g\equiv0$.

For some special functions $G$, \eqref{ass:1-1} holds under weaker conditions like e.g.
\begin{equation}\label{ass:def_uniform}
\frac1{\psi^{1/p_1}}\,\in\,L_{p_1^*}([0,\infty)).
\end{equation}
We illustrate this for $G(y)=1-\re^{-y}$ and $G(y)=1-\cos(y)$.
It is enough to consider ${\setu}=\{1\}$ in \eqref{ass:1-1}, i.e. to show that
the operator $K$ given by
\[
 (Kg)(x)\,:=\,\int_{D}g(t)\,
\kappa(x,t)\rd t
\]
satisfies $g=0$ almost everywhere  whenever $Kg=0$ and $g \in L_{p_1,\psi}$.

Indeed, using H\"older's inequality and \eqref{ass:def_uniform}, we 
get that $L_{p_1,\psi} \subseteq L_1([0,\infty))$. 
Hence it is enough to show that
$Kg = 0$ for some $g\in L_1([0,\infty))$ implies $g=0$ in $L_1([0,\infty))$.

For $G(y)=1-\re^{-y}$, i.e. 
\[
 (Kg)(x)\,:=\,\int_0^\infty g(t)\, (1-\re^{-xt}) \rd t,
\]
this follows from the properties of the Laplace transform ${\cal L}$. Indeed,
observe that with $c=\int_0^\infty g(t) \rd t$, we have
\[ 
 0 = (Kg)(x) = c- ({\cal L}g)(x) \quad\mbox{for all\ }x\in [0,\infty).
\]
Hence ${\cal L}g=c$ is a constant, which is only possible if this constant is $c=0$.
But then ${\cal L}g=0$, and $g=0$ almost everywhere follows from the injectivity property of the Laplace transform. 

For $G(y)=1-\cos(y)$, i.e. 
\[
 (Kg)(x)\,:=\,\int_0^\infty g(t)\, (1-\cos{(xt)}) \rd t,
\]
we can argue similarly with the Fourier transform instead of the Laplace transform by
extending $g$ to an even function on $(-\infty,\infty)$.
\end{example}

\subsection{The Function Approximation Problem}\label{secalgdef}
 
We follow \cite{Was14}. Let $\omega$ be a probability density on $D$ and
let $q\in[1,\infty]$. For non-empty $\setu$, let
$L_{\setu,q,\omega}=L_{\setu,q,\omega}(D^{|\setu|})$ be the space of functions
with finite semi-norm
\[
\|f\|_{L_{\setu,q,\omega}}\,=\,\left(\int_{D^{|\setu|}}
|f(\bst_\setu)|^{q}
\,\omega_\setu(\bst_\setu)\rd\bst_\setu\right)^{1/q}
\quad\mbox{with}\quad
\omega_\setu(\bst_\setu)\,=\,\prod_{j\in\setu}\omega(t_j).
\]
For $\setu=\emptyset$, the corresponding space is
$L_{\emptyset,q,\omega}$, the space of constant functions.

Consider next the embedding operators
\[
S_\setu(f_\setu)\,=\,f_\setu\,\in\,L_{\setu,q,\omega}.
\]
For them to be well defined, we assume that
\[
\widetilde{\kappa}_{q,p_1,\omega}\,:=\,
\|\widehat{\kappa}_{p_1}\|_{L_{\{1\},q,\omega}}\,=\,
\left(\int_D|\widehat{\kappa}_{p_1}(x)|^{q}\,\omega(x)
\rd x\right)^{1/q}\,<\,\infty.
\]
Then for any $f_\setu\in F_\setu$ we have
\begin{equation}\label{help}
\|f_{\setu}\|_{L_{\setu,q,\omega}}\,\le\,\|f_\setu\|_{F_\setu}\,
\widetilde{\kappa}_{q,p_1,\omega}^{|\setu|},
\quad\mbox{i.e.,}\quad
\|S_\setu\|\,\le\,\widetilde{\kappa}_{q,p_1,\omega}^{|\setu|}. 
\end{equation}
This means that \eqref{ass:tnspr} holds with
\[C_1\,=\,\widetilde{\kappa}_{q,p_1,\omega}.
\]
Of course, $C_1=\widetilde{\kappa}_{q,p_1,\omega}$ depends also on $\psi$.

Let $\calL_{s,q,\omega}$ be a space containing $\calF_{s,p_1,p_2,\bsgamma}$
and endowed with a semi-norm such that for every $\setu$ and
$f_\setu\in F_\setu$
\[
\|f_\setu\|_{\calL_{s,q,\omega}}\,\le\,\|f_\setu\|_{L_{\setu,q,\omega}}.
\]

Finally, let $\calS_s$ be the embedding operator
\[
\calS_s\,:\,\calF_{s,p_1,p_2,\bsgamma}\,\to\,
\calL_{s,q,\omega}.
\]
Of course, it depends on all the parameters, $p_1,p_2,q,\psi,\omega$, and
the weights $\bsgamma$.
We assume that these parameters satisfy the following condition
\[
  \left(\sum_{\setu \subseteq [s]} \left(\gamma_\setu\,
  \widetilde{\kappa}_{q,p_1,\omega}^{|\setu|}\right)^{p_2^*}\right)^{1/p_2^*}
  \,<\,\infty  
\]
since then
\[\|\calS_s\|\,\le\,\left(\sum_{\setu \subseteq [s]} \left(\gamma_\setu\,
\widetilde{\kappa}_{q,p_1,\omega}^{|\setu|}\right)^{p_2^*}\right)^{1/p_2^*}.
\]
Note that for product weights the embedding operator is of tensor product form.

We illustrate the assumptions above for the examples from the previous section.

\begin{example}
We continue Example \ref{exmp1} here. Consider $\omega\equiv1$.
This case was studied in \cite{KrPiWa16}. We have
\[\|S_1\|\,\le\,
\widetilde{\kappa}_{q,p_1,\omega}\,=\,\left\{\begin{array}{ll}
1 &\mbox{if $q=\infty$ or $p_1=1$},\\
(1+q/p_1^*)^{-1/q} &\mbox{otherwise}.
\end{array}\right.
\]
\end{example}

\begin{example}
We return to Example \ref{exmp3} and assume 
that $\omega(x)=\mu \re^{-\mu x}$ for some $\mu>0$. 
In what follows and some other places we use the well known fact
that
\[\int_0^\infty x^a\,\re^{-b\,x}\rd x\,=\,\frac{\Gamma(a+1)}{b^{a+1}}
\quad\mbox{for}\quad a,b>0.
\]

We begin with the case of $p_1=1$. It is easy to see that
\[\widetilde{\kappa}_{q,p_1,\omega}\,=\,\left\{
\begin{array}{ll}\frac1{(r-1)!\,\mu^{r-1}}\,(\Gamma((r-1)q+1))^{1/q}
  &\mbox{if\ }q<\infty\\
  \infty&\mbox{if\ }q=\infty,\end{array}\right.
\]
for $\lambda\ge0$. For $\lambda<0$,
$\widetilde{\kappa}_{\infty,p_1,\omega}=\infty$  if
$q=\infty$ or $\mu+\lambda q/p_1\le0$. Otherwise,
for $\lambda<0$, 
\[\widetilde{\kappa}_{q,p_1,\omega}\,\le\,
\left(\frac{r-1}{|\lambda|}\right)^{r-1}\,
\left(\re^{-\mu(r-1)/|\lambda|}+
\re^{-(r-1)/|\lambda|}\right)^{1/q}. 
\]

We now consider the case of $p_1>1$. For $q=\infty$ and any $\lambda$ we
have
\begin{eqnarray*}(r-1)!\,\widetilde{\kappa}_{\infty,p_1,\omega}
&=&\sup_{x\ge0} \left(\int_0^x(x-t)^{(r-1)p_1^*}\,
\re^{-\lambda t (p_1^*-1)}\rd t\right)^{1/p_1^*}\\
&\ge& \sup_{x\ge1} \left(\int_0^1(x-t)^{(r-1)p_1^*}\,
\re^{-\lambda t (p_1^*-1)}\rd t\right)^{1/p_1^*}\,=\,\infty.
\end{eqnarray*}
Therefore, for the rest of this example, we consider $q<\infty$.

If $\lambda\ge0$, then 
\[\widetilde{\kappa}_{q,p_1,\omega}\,\le\,
\frac{\left(\Gamma((r-1/p_1)q+1)\right)^{1/q}}
{(r-1)!\,((r-1)p_1^*+1)^{1/p_1^*}\,\mu^{r-1/p_1}}\,
\]
Consider next $\lambda<0$. Since
\[(r-1)!\,\widehat{\kappa}_{p_1}(x)\,\ge\,\left(\int_0^{x-1}
\re^{\lambda\,t\,(p_1^*-1)}\rd t\right)^{1/p_1^*},
\]
we conclude that
\[\widetilde{\kappa}_{q,p_1,\omega}\,=\,\infty\quad
\mbox{if}\quad \mu+\lambda\,\frac{q}{p_1}\,\le\,0. 
\]
If $\mu+\lambda q/p_1>0$, then $\widetilde{\kappa}_{q,p_1,\omega}$
is bounded from above by 
\[
\frac{\left(\Gamma((r-1)q+1)\right)^{1/q}}
{(r-1)!\,(|\lambda|\,(p_1^*-1))^{1/p_1^*}\,\mu^{1/q}\,
(\mu+\lambda q/p_1)^{r-1+1/q}}\,
\]
and
\[\frac{\left(\Gamma((r-1)\,q+2)\right)^{1/q}}
{(r-1)!\,((r-1)p_1^*+1)^{1/p_1^*}\,\mu^{1/q}\,
(\mu+\lambda q/p)^{r-1+2/q}}.
\]
\end{example}

\bigskip

Let the assumptions from the previous section be satisfied.

\begin{remark}\label{def:trd} \rm
In this setting the $\varepsilon$-truncation dimension from Definition~\ref{def:truncdim} is the smallest natural number $k$ such that
\[
\left\|\sum_{\setu\not\subseteq[k]}f_\setu\right\|_{\calL_{s,q,\omega}}
\le\,\e\,\left\|\sum_{\setu\not\subseteq[k]}f_\setu
\right\|_{\calF_{s,p_1,p_2,\bsgamma}}\quad\mbox{for all}\ 
f=\sum_{\setu \subseteq [s]} f_{\setu} \in\calF_{s,p_1,p_2,\bsgamma}. 
\]
\end{remark}

We obtain the following corollary of Proposition~\ref{prop:simple}
and Theorem~\ref{thm:simple}.

\begin{corollary} We have
\begin{equation}\label{k-e}
\dimtr(\e)\le\min\left\{k\ :\
\left(\sum_{\setu\not\subseteq[k]}(\gamma_\setu\,
\widetilde{\kappa}_{q,p_1,\omega}^{|\setu|})^{p_2^*}\right)^{1/p_2^*}
\,\le\, \e\right\}
\end{equation}
which reduces to
$\dimtr(\e)\le\min\left\{k\ :\ \sup_{\setu\not\subseteq[k]}\gamma_\setu\,\widetilde{\kappa}_{q,p_1,\omega}^{|\setu|}\,\le\, \e\right\}$
for $p_2=1$. 

For product weights, ${\rm dim}^{\rm trnc}(\e)$
is bounded from above by 
\[\min\left\{k\,:\, \prod_{j=1}^s(1+(\gamma_j\,
\widetilde{\kappa}_{q,p_1,\omega})^{p_2^*})^{1/p_2^*}\,
\left(1-\re^{-\sum_{j=k+1}^s (\gamma_j 
  \widetilde{\kappa}_{q,p_1,\omega})^{p_2^*}}\right)^{1/p_2^*}\le
\e\right\}.
\]
\end{corollary}

For $k \le s$ let 
\[
 \calF_{k,p_1,p_2,\bsgamma}=\bigoplus_{\setu\subseteq[k]}F_\setu
\]
be the subspace of $\calF_{s,p_1,p_2,\bsgamma}$ consisting of $k$-variate 
functions $f([\cdot_{[k]};\bszero_{-[k]}])$, and let 
$A_{k,n}$ be an algorithm for approximating functions from 
$\calF_{k,p_1,p_2,\bsgamma}$ that uses $n$ function values. The worst case error of $A_{k,n}$ with respect to the space 
$\calF_{k,p_1,p_2,\bsgamma}$ is 
\[
  e(A_{k,n};\calF_{k,p_1,p_2,\bsgamma}):=
\sup_{f\in \calF_{k,p_1,p_2,\bsgamma}}\frac{\|f-A_{k,n}(f)\|_{\calL_{k,q,\omega}}}
{\|f\|_{\calF_{k,p_1,p_2,\bsgamma}}}.
\]
Now let 
\[
  \calA^{\rm trnc}_{s,k,n}(f)\,=\,A_{k,n}(f([\cdot_{[k]};\bszero_{-[k]}]))
\]
be an algorithm for  approximating functions from the whole space
$\calF_{s,p_1,p_2,\bsgamma}$. The worst case error of $\calA^{\rm trnc}_{s,k,n}$ is defined as 
\[
e(\calA^{\rm trnc}_{s,k,n};\calF_{s,p_1,p_2,\bsgamma})\,:=\,
\sup_{f\in\calF_{s,p_1,p_2,\bsgamma}}
\frac{\|f-\calA^{\rm trnc}_{s,k,n}(f)\|_{\calL_{s,q,\omega}}}
{\|f\|_{\calF_{s,p_1,p_2,\bsgamma}}}.
\]

This yields the following corollary of Theorem~\ref{thm:main}.

\begin{corollary}
For given $\e>0$ and  $k\ge \dimtr(\e)$ we have 
\[
e(\calA^{\rm trnc}_{s,k,n};\calF_{s,p_1,p_2,\bsgamma})\,\le\,
\left(\e^{p_2^*}+e(A_{k,n};\calF_{k,p_1,p_2,\bsgamma})^{p_2^*}\right)^{1/p_2^*}
\]
which reduces to 
$e(\calA_{s,k,n}^{\rm trnc};\calF_{s,p_1,1,\bsgamma})
\le\max\left(\e\,,\,e(A_{k,n};\calF_{k,p_1,1,\bsgamma})\right)$
for $p_2=1$.
  
\end{corollary}

\subsection{The Integration Problem}\label{secint}

In this subsection we assume that \eqref{ass:inft} is satisfied.
We consider the problem of numerically approximating the integral 
\[
\calI_s (f)=\int_{D^s} f(\bsx) \omega_{[s]} (\bsx)\rd\bsx,
\]
where $f\in \calF_{s,p_1,p_2,\bsgamma}$, where $\omega$ is a
probability density function on $D$, and
$\omega_{\setu}(\bsx_{\setu})=\prod_{j\in\setu} \omega (x_j)$
for $\setu\subseteq [s]$. 

We require in this section that $\overline{\kappa}_{p_1,\omega}$, defined by
\[
 \overline{\kappa}_{p_1,\omega} (t):=\int_{D} 
 \frac{|\kappa (x,t)|}{\psi ^{1/p_1} (t)}\omega(x)\rd x,
\]
is such that $\norm{\overline{\kappa}_{p_1,\omega}}_{L_{p_1^*}}<\infty$.

Let now $f\in \calF_{s,p_1,p_2,\bsgamma}$. 
For non-empty $\setu$, let $g_\setu\in L_{\setu,p_1,\psi}$ be such that 
$\norm{f_\setu}_{F_\setu}=\norm{g_\setu}_{L_{\setu,p_1,\psi}}$ (as outlined in Section \ref{secgeneral}). 
We then have
\begin{eqnarray*}
\abs{\calI_s (f_\setu)} &=& \abs{\int_{D^{\abs{\setu}}} f_\setu (\bsx_\setu)\omega_\setu (\bsx_\setu)\rd\bsx_\setu}
 =\abs{\int_{D^{\abs{\setu}}} g_\setu (\bst_\setu) \psi_\setu^{1/p_1} (\bst_\setu) \prod_{j\in\setu} \overline{\kappa}_{p_1,\omega}  (t_j) \rd\bst_\setu}\\
 &\le& \norm{g_\setu}_{L_{\setu,p_1,\psi}}
 \norm{\overline{\kappa}_{p_1,\omega}}_{L_{p_1^*}}^{\abs{\setu}}
 = \norm{f_\setu}_{F_\setu}
 \norm{\overline{\kappa}_{p_1,\omega}}_{L_{p_1^*}}^{\abs{\setu}}.
\end{eqnarray*}
Since H\"older's inequality is sharp we conclude that
\[
\|I_\setu\|\,=\,\norm{\overline{\kappa}_{p_1,\omega}}_{L_{p_1^*}}^{|\setu|},
\]
where $I_\setu$ is the restriction of $\calI_s$ to $F_\setu$. 
This means that \eqref{ass:tnspr} hold with equality for
\[
C_1\,=\,\norm{\overline{\kappa}_{p_1,\omega}}_{L_{p_1^*}}. 
\]

\begin{example}
Let us once more return to Example \ref{exmp3} with
$\omega(x)=\mu\,\re^{-\mu\,x}$. Then the $L_{p_1^*}$-norm
of $\overline{\kappa}_{p_1,\omega}$ is given by 
\[\frac1{(r-1)!}\left(\int_0^\infty\re^{-\lambda\,t\,p_1^*/p_1}
\left(\mu \int_t^\infty(x-t)^{r-1}\,\re^{-\mu\,x}\rd x\right)^{p_1^*}
\rd t\right)^{1/p_1^*}.
\]
The inner integral, after the change of variables $z=x-t$, is equal to
\[\mu \int_0^\infty z^{r-1}\,\re^{-\mu(z+t)}\rd z\,=\,
\frac1{\mu^{r-1}}\,\re^{-\mu\,t}\,\Gamma(r)
\]
and, therefore,
\[\|\overline{\kappa}_{p_1,\omega}\|_{L_{p_1^*}}\,=\,
\left\{\begin{array}{ll}\frac{\Gamma(r)}{(r-1)!\,\mu^{r-1}}\left(
\frac1{p_1^*\,(\lambda/p_1+\mu)}\right)^{1/p^*} &\mbox{if\ }
\mu+\lambda/p_1>0,\\
\infty &\mbox{if\ }\mu+\lambda/p_1\le0
\end{array}\right.
\]
with the convention that $(1/(p_1^*(\lambda+\mu))^{1/p_1^*}=1$ for
$p_1^*=\infty$. 
\end{example}

Let the assumptions from the previous section be satisfied.

\begin{remark}\label{def:trdint} \rm
In this setting the $\varepsilon$-truncation dimension from Definition~\ref{def:truncdim} is the smallest natural number $k$ such that
\[
\abs{\sum_{\setu\not\subseteq[k]} \calI_s (f_\setu)}
\le\,\e\,\left\|\sum_{\setu\not\subseteq[k]}f_\setu
\right\|_{\calF_{s,p_1,p_2,\bsgamma}}\quad\mbox{for all}\ 
f=\sum_{\setu \subseteq [s]} f_{\setu} \in\calF_{s,p_1,p_2,\bsgamma}. 
\]
\end{remark}

We obtain the following corollary of Proposition~\ref{prop:simple}
and Theorem~\ref{thm:simple}.

\begin{corollary} We have
\begin{equation}\label{k-eint}
\dimtr(\e)\,\le\,\min\left\{k\ :\
\left(\sum_{\setu\not\subseteq[k]}(\gamma_\setu\,
\norm{\overline{\kappa}_{p_1,\omega}}_{L_{p_1^*}}^{|\setu |})^{p_2^*}\right)^{1/p_2^*} \le \e\right\}.
\end{equation}
which reduces to
$\dimtr(\e)\le\min\left\{k : 
\sup_{\setu\not\subseteq[k]}\gamma_\setu
\norm{\overline{\kappa}_{p_1,\omega}}_{L_{p_1^*}}^{|\setu |} \le \e\right\}$
for $p_2=1$.

For product weights ${\rm dim}^{\rm trnc}(\e)$ is bounded from above by
the smallest $k$ for  which
\[
\prod_{j=1}^s(1+(\gamma_j\norm{\overline{\kappa}_{p_1,\omega}}_{L_{p_1^*}})^{p_2^*})^{1/p_2^*}\,\left(1-\re^{-\sum_{j=k+1}^s 
(\gamma_j\norm{\overline{\kappa}_{p_1,\omega}}_{L_{p_1^*}})^{p_2^*}}
\right)^{1/p_2^*}\, \le\, \e.
\]
\end{corollary}

For $k \le s$ let 
$A_{k,n}$ be an algorithm for integrating functions from 
$\calF_{k,p_1,p_2,\bsgamma}$ that uses $n$ function values. The worst case error of $A_{k,n}$ with respect to the space 
$\calF_{k,p_1,p_2,\bsgamma}$ is 
\[
  e(A_{k,n};\calF_{k,p_1,p_2,\bsgamma}):=
\sup_{f\in \calF_{k,p_1,p_2,\bsgamma}}\frac{\abs{\calI_k (f)-A_{k,n}(f)}}
{\|f\|_{\calF_{k,p_1,p_2,\bsgamma}}}.
\]
Now let 
\[
  \calA^{\rm trnc}_{s,k,n}(f)\,=\,A_{k,n}(f([\cdot_{[k]};\bszero_{-[k]}]))
\]
be an algorithm for integrating functions from the whole space
$\calF_{s,p_1,p_2,\bsgamma}$. The worst case error of $\calA^{\rm trnc}_{s,k,n}$ is defined as 
\[
e(\calA^{\rm trnc}_{s,k,n};\calF_{s,p_1,p_2,\bsgamma})\,:=\,
\sup_{f\in\calF_{s,p_1,p_2,\bsgamma}}
\frac{|\calI_s (f)-\calA^{\rm trnc}_{s,k,n}(f)|}
{\|f\|_{\calF_{s,p_1,p_2,\bsgamma}}}.
\]

This yields the following corollary of Theorem~\ref{thm:main}.

\begin{corollary}\label{thm:mainint}
For given $\e>0$ and  $k\ge \dimtr(\e)$ we have 
\[
e(\calA^{\rm trnc}_{s,k,n};\calF_{s,p_1,p_2,\bsgamma})\,\le\,
\left(\e^{p_2^*}+e(A_{k,n};\calF_{k,p_1,p_2,\bsgamma})^{p_2^*}\right)^{1/p_2^*}
\]
which reduces to
$e(\calA^{\rm trnc}_{s,k,n};\calF_{s,p_1,1,\bsgamma})\,\le\,
\max\left(\e\,,\,e(A_{k,n};\calF_{k,p_1,1,\bsgamma}\right)$
for $p_2=1$.
\end{corollary}

\section{Unanchored Spaces of Multivariate Functions}
Let $\kappa$, $\omega$, and $\calI_s$ be as in the previous section. 
Also here we assume that
$\|\overline{\kappa}_{p_1,\omega}\|_{L_{p^*}}<\infty$.
In what follows we use $I$ to denote the $\omega$-weighted
integral operator for univariate functions,
\[I(f)\,=\,\int_Df(x)\,\omega(x)\rd x.
\]
Of course $\|I\|=\|\overline{\kappa}_{p_1,\omega}\|_{L_{p^*_1}}$.

Consider 
\[
\kappa_{\setu,\omega}(\bsx_\setu,\bst_\setu)\,=\,\prod_{j\in\setu}
\left(\kappa(x_j,t_j)-I(\kappa(\cdot,t_j))\right)
\]
and
\[
K_{\setu,\omega}(g_\setu)(\bsx_\setu)\,=\,
\int_{D^{|\setu|}} g_\setu(\bst_\setu)\,
\kappa_{\setu,\omega}(\bsx_\setu,\bst_\setu)\rd\bst_\setu
\]
and the corresponding space $\calF_{s,p_1,p_2,\bsgamma,\omega}$
of functions
\[f(\bsx)\,=\,\sum_\setu f_{\setu,\omega}(\bsx_{\setu})\quad\mbox{with}\quad
f_{\setu,\omega}(\bsx_\setu)\,=\,K_{\setu,\omega}(g_\setu)(\bsx_{\setu})
\]
such that 
\[\|f\|_{\calF_{s,p_1,p_2,\bsgamma,\omega}}\,:=\,\left(
\sum_\setu\gamma_\setu^{-p_2}\,\|g_\setu\|^{p_2}_{L_{\setu,p_1,\psi}}
\right)^{1/p_2}\,<\,\infty. 
\]
Instead of being anchored, the functions $f_{\setu,\omega}$ satisfy
the following property
\[\int_D f_{\setu,\omega}(\bsx_\setu)\,\prod_{k\in\setu}
\omega(x_k)\rd x_j\,=\,0\quad\mbox{if}\quad   j\in\setu. 
\]

As in \cite{HeRiWa15}, one can show that the spaces
$\calF_{s,p_1,p_2,\bsgamma}$ and $\calF_{s,p_1,p_2,\bsgamma,\omega}$
as sets of functions are equal if and only if
\begin{equation}\label{weights}
  \gamma_\setu\,>\,0\quad\mbox{implies that}\quad
  \gamma_\setv\,>\,0\quad\mbox{for all}\quad\setv\subseteq\setu.
\end{equation}
From now on, we assume that \eqref{weights} is satisfied. Of course
\eqref{weights} always holds true for product weights. 

Let $\imath_{p_1,p_2}$ be the embedding
\[
\imath_{p_1,p_2}:\calF_{s,p_1,p_2,\bsgamma}\to\calF_{s,p_1,p_2,\bsgamma,\omega}
\quad\mbox{and}\quad \imath_{p_1,p_2}(f)\,=\,f,
\]
and let $\imath^{-1}_{p_1,p_2}$ be its inverse. 
As in \cite{KrPiWa16a}, see also \cite{GnHeHiRiWa16}, 
one can check that
\[\|\imath_{p_1,p_2}\|\,=\,\|\imath^{-1}_{p_1,p_2}\|.
\]
Moreover, following the approach in \cite{GnHeHiRiWa16},
one can provide exact formulas for the norms of the embeddings for
$p_1,p_2\in\{1,\infty\}$ and next, using interpolation theory
(as in \cite{SH}, see also \cite{GnHeHiRiWa16}), derive
upper bounds for arbitrary values of $p_1$ and $p_2$.

More precisely, we have the following proposition.

\begin{proposition}\label{prop:ANOVA}
Suppose that $\|\overline{\kappa}_{p_1,\omega}\|_{L_{p_1^*}}<\infty$
  for $p_1\in\{1,\infty\}$. Then
\[\|\imath_{p_1,p_2}\|\,=\,\left\{\begin{array}{ll}
\displaystyle
\max_\setu\,\sum_{\setv\subseteq\setu}\frac{\gamma_\setu}{\gamma_\setv}\,
\|\overline{\kappa}_{1,\omega}\|_{L_{\infty}}^{|\setu|-|\setv|}
&
\mbox{if $p_1=1$ and $p_2=1$},\\ 
\displaystyle
\max_\setu\sum_{\setv\subseteq[s]\setminus\setu}
\frac{\gamma_{\setu\cup\setv}}{\gamma_\setu}\,
\|\overline{\kappa}_{1,\omega}\|_{L_\infty}^{|\setv|}
& \mbox{if $p_1=1$   and $p_2=\infty$,}\\ 
\displaystyle
\max_\setu\sum_{\setv\subseteq\setu}\frac{\gamma_\setu}{\gamma_\setv}\,
\|\overline{\kappa}_{\infty,\omega}\|_{L_1}^{|\setu|-|\setv|}
&\mbox{if $p_1=\infty$ and $p_2=1$,}\\ 
\displaystyle
\max_\setu \sum_{\setv\subseteq[s]\setminus\setu}\frac{\gamma_{\setu\cup\setv}}
    {\gamma_\setu}\,
\|\overline{\kappa}_{\infty,\omega}\|_{L_1}^{|\setv|}
&\mbox{if $p_1=\infty$ and $p_2=\infty$.}
\end{array}\right.
\]
\end{proposition}

To give a flavor of the proof, we prove the proposition for $p_1=p_2=1$.
\begin{proof}
For $f=\sum_\setu f_{\setu,\omega}$ we have
\begin{eqnarray*}
  f_{\setu,\omega}&=&K_{\setu,\omega}(g_\setu)\,=\,\int_{D^{|\setu|}}
  g_\setu(\bst_\setu)\,\prod_{j\in\setu}(\kappa(\cdot_j,t_j)
  -\,I(\kappa(\cdot,t_j)))
  \rd\bst_\setu\\
  &=&\sum_{\setv\subseteq\setu}\int_{D^{|\setu|}} g_\setu(\bst_\setu)\,
  \kappa_\setv(\cdot_\setv,\bst_\setv)\,\prod_{j\in\setu\setminus\setv}
  (-1)\,\,I(\kappa(\cdot,t_j))
  \rd\bst_\setu.
\end{eqnarray*}
Therefore
\begin{eqnarray*}
  f&=&\sum_\setu\sum_{\setv\subseteq\setu}
  \int_{D^{|\setu|}} g_\setu(\bst_\setu)\,
  \kappa_\setv(\cdot_\setv,\bst_\setv)\,\prod_{j\in\setu\setminus\setv}
  (-1)\,\,I(\kappa(\cdot,t_j))
  \rd\bst_\setu\\
  &=&\sum_\setv\int_{D^{|\setv|}}\kappa_\setv(\cdot_\setv ,\bst_\setv)
  \sum_{\setw,\setw\cap\setv=\emptyset}\int_{D^{|\setw|}}
  g_{\setv\cup\setw}(\bst_\setv,\bst_\setw)\,\prod_{j\in\setw}
  (-1)\,\,I(\kappa(\cdot,t_j))\rd\bst_\setw\rd\bst_\setv,
\end{eqnarray*}
where $(\bst_\setv,\bst_\setw)=\bst_{\setv \cup \setw}$, which implies that $f\,=\,\sum_\setv K_{\setv}(h_\setv)$ 
with 
\[h_{\setv}(\bst_\setv)\,=\,\sum_{\setw,\setw\cap\setv=\emptyset}
\int_{D^{|\setw|}}
  g_{\setv\cup\setw}(\bst_\setv,\bst_\setw)\,\prod_{j\in\setw}
  (-1)\,I(\kappa(\cdot,t_j))
  \rd\bst_\setw. 
\]
Clearly  
\[\|h_\setv\|_{L_{\setv,1,\psi}}\,\le\,
\sum_{\setw,\setw\cap\setv=\emptyset}\|g_{\setv\cup\setw}\|_{L_{\setv\cup\setw,1,\psi}}
\,
\|\overline{\kappa}_{1,\omega}\|_{L_\infty}^{|\setw|}
\]
and using $\setu=\setv\cup\setw$ we get 
\begin{eqnarray*}
\sum_{\setv}\gamma_\setv^{-1}\,\|h_\setv\|_{L_{\setv,1,\psi}}
&\le&\sum_\setu \gamma_\setu^{-1}\,\|g_\setu\|_{L_{\setu,1,\psi}}\,
\gamma_\setu\sum_{\setv\subseteq\setu}\gamma_\setv^{-1}\,
\|\overline{\kappa}_{1,\omega}\|_{L_\infty}^{|\setu|-|\setv|}
\\  &\le&
\|f\|_{\calF_{s,1,1,\bsgamma,\omega}}\,\max_\setu \sum_{\setv\subseteq\setu}
\frac{\gamma_\setu}{\gamma_\setv}\,
\|\overline{\kappa}_{1,\omega}\|_{L_\infty}^{|\setu|-|\setv|}.
\end{eqnarray*}
This proves the bound on $\|\imath^{-1}\|$. Since the H\"older
inequality is sharp, we actually have equality.
The proof for $\imath$ is identical.
\end{proof}

For product weights $\gamma_\setu=\prod_{j\in\setu}\gamma_j$
the expressions in the proposition above reduce to
\[
\prod_{j=1}^s\left(1+\gamma_j\,
\|\overline{\kappa}_{1,\omega}\|_{L_\infty}\right)
\quad\mbox{if $p_1=1$}
\]
and to
\[
\prod_{j=1}^s\left(1+\gamma_j\,
\|\overline{\kappa}_{\infty,\omega}\|_{L_1}\right)
\quad\mbox{if $p_1=\infty$}.
\]

Applying interpolation theory we get, as in \cite{GnHeHiRiWa16}:

\begin{corollary}\label{wniosek}
Suppose that $\|\overline{\kappa}_{p_1,\omega}\|_{L_{p_1^*}}<\infty$ 
for any $p_1\in[1,\infty]$.
If $p_1\le p_2$ then
\[
\|\imath_{p_1,p_2}\|\,\le\,\|\imath_{1,\infty}\|^{1/p_1-1/p_2}\,
\|\imath_{1,1}\|^{1/p_2}\,\|\imath_{\infty,\infty}\|^{1-1/p_1},
\]
and if $p_2< p_1$ then
\[
\|\imath_{p_1,p_2}\|\,\le\,\|\imath_{\infty,1}\|^{1/p_2-1/p_1}\,
\|\imath_{1,1}\|^{1/p_1}\,\|\imath_{\infty,\infty}\|^{1-1/p_2}.
\]
For product weights we have
\[\|\imath_{p_1,p_2}\|\,\le\,\prod_{j=1}^s\left(1+\gamma_j\,
\|\overline{\kappa}_{1,\omega}\|_{L_\infty}\right)^{1/p_1}\,
\left(1+\gamma_j\,
\|\overline{\kappa}_{\infty,\omega}\|_{L_1}\right)^{1-1/p_1}.
\]
\end{corollary}
It was shown in \cite{SH} for product weights and in \cite{KrPiWa16a}
for a number of different types of weights that the upper bounds
in Corollary \ref{wniosek} are sharp.

Suppose now that $\sum_{j=1}^\infty \gamma_j\,<\,\infty$. 
Then the norms of the embeddings are uniformly bounded,
\[\|\imath_{p_1,p_2}\|\,\le\,\prod_{j=1}^\infty\left(1+\gamma_j\,
\|\overline{\kappa}_{1,\omega}\|_{L_\infty}\right)^{1/p_1}\,
\left(1+\gamma_j\,
\|\overline{\kappa}_{\infty,\omega}\|_{L_1}\right)^{1-1/p_1},
\]
for any $s$ including $s=\infty$. Hence the results of previous sections
are applicable for unanchored  spaces considered in this section. 

\begin{remark}
It is possible to consider even more general unanchored spaces. 
Indeed, consider a linear functional $\ell$ that is continuous
for the space of univariate functions, i.e., with
\[
\|\overline{\ell}_{p_1}\|_{L_{p_1^*}}\,<\,\infty, \quad
\mbox{where}\quad
\overline{\ell}_{p_1}(t)\,:=\,\frac{\ell(\kappa(\cdot,t))}{\psi^{1/p_1}}.
\]
Suppose also that
\[\ell\left(\int_D g(t)\,\kappa(\cdot,t)\rd t\right)\,=\,
\int_D g(t)\,\ell(\kappa(\cdot,t))\rd t\quad
\mbox{for all\ }g\in L_{p_1,\psi}(D).
\]
For nonempty $\setu$, define 
\[K_{\setu,\ell}(g_\setu)(\bsx_\setu)\,:=\,\int_{D^{|\setu|}}g_\setu(\bst_\setu)
\prod_{j\in\setu}(\kappa (x_j,t_j)-\ell(\kappa(\cdot_j;t_j))) \rd \bst_\setu.
\]
Then the corresponding functions $f_{\setu,\ell}=K_{\setu,\ell}(g_\setu)$
satisfy
\[\ell_j(f_{\setu,\ell})\,=\,0\quad\mbox{if\ }j\in\setu.
\]
Here $\ell_j$ denotes the functional $\ell$ acting on functions with respect
to the  $j^{{\rm th}}$ variable.
More formally,
\[
\ell_j\,=\,\bigotimes_{n=1}^s T_n\quad\mbox{with}\quad
T_n\,=\,\left\{\begin{array}{ll} {\rm id} &\mbox{if\ }n\not=j,\\
\ell &\mbox{if\ }n=j,\end{array}\right.
\]
where ${\rm id}$ is an identity operator. 
For instance for $\ell(g)=g(0)+\int_D g(t)\rd t$,
\[\ell_j(f)(\bsx)\,=\,f([\bsx_\setu;\bszero_{-\setu}])+
\int_D f(\bsx)\rd x_j\quad\mbox{with}\quad \setu=\{j\}.
\]

Let $\calF_{s,p_1,p_2,\bsgamma,\ell}$ be the
Banach space of functions
$f=\sum_{\setu, \gamma_\setu>0} K_{\setu,\ell}(g_\setu)$ 
with the norm 
\[
\|f\|_{\calF_{s,p_1,p_2,\bsgamma,\ell}}\,=\,
\left(\sum_{\setu,\gamma_\setu>0}\gamma_\setu^{-1/p_2}\,
\|g_\setu\|_{L_{\setu,p_1,\psi}}^{p_2}\right)^{1/p_2}. 
\]
It is easy to extend all the results of this section provided that
$\|\overline{\ell}_{p_1}\|_{L_{p_1^*}}$ is finite for all $p_1$.
In particular, Proposition \ref{prop:ANOVA} and Corollary
\ref{wniosek} hold with $\|\overline{\kappa}_{p_1,\omega}\|_{L_{p_1^*}}$
replaced by $\|\overline{\ell}_{p_1}\|_{L_{p_1^*}}$. 
\end{remark}

\paragraph{Acknowledgment.} The authors would like to thank two anonymous referees for their suggestions for improving the paper.

\begin{small}
\noindent\textbf{Authors' addresses:}\\

\medskip

 \noindent Aicke Hinrichs\\
 Institute for Analysis\\
 Johannes Kepler University Linz\\
 Altenbergerstr. 69, 4040 Linz, Austria.\\
 \texttt{Aicke.Hinrichs@jku.at}\\
 
 \medskip
 
 \noindent Peter Kritzer\\ 
 Johann Radon Institute for Computational and Applied Mathematics (RICAM)\\
 Austrian Academy of Sciences\\
 Altenbergerstr. 69, 4040 Linz, Austria.\\
 \texttt{peter.kritzer@oeaw.ac.at}\\
 
 \medskip
 
 \noindent Friedrich Pillichshammer\\
 Department of Financial Mathematics and Applied Number Theory\\
 Johannes Kepler University Linz\\
 Altenbergerstr. 69, 4040 Linz, Austria\\
 \texttt{friedrich.pillichshammer@jku.at}\\
 
 \medskip
 
 \noindent G.W. Wasilkowski\\
 Computer Science Department\\
 University of Kentucky\\
 301 David Marksbury Building, Lexington, KY 40506, USA\\  
 \texttt{greg@cs.uky.edu}

 \end{small}

\end{document}